\documentclass{article}%

\usepackage{amsmath,amssymb,amsfonts}%
\usepackage{theorem}%
\usepackage{color}%
\usepackage{hyperref}%

\theoremstyle{change}%

\newtheorem{definition}{Definition:}[section]%
\newtheorem{theorem}[definition]{Theorem:}%
\newtheorem{proposition}[definition]{Proposition:}%
\newtheorem{lemma}[definition]{Lemma:}%
\newtheorem{corollary}[definition]{Corollary:}%
{\theorembodyfont{\rmfamily} \newtheorem{remark}[definition]{Remark:}}%
{\theorembodyfont{\rmfamily} }%

\newenvironment{proof}
  {{\bf Proof:}}
  {\qquad \hspace*{\fill} $\Box$}%
  
\newcommand{\inner}{\operatorname{int}}%
\newcommand{\cl}{\operatorname{cl}}%
\newcommand{\diam}{\operatorname{diam}}%
\newcommand{\CC}{\mathcal{C}}%
\newcommand{\OC}{\mathcal{O}}%
\newcommand{\QC}{\mathcal{Q}}%
\newcommand{\SC}{\mathcal{S}}%
\newcommand{\UC}{\mathcal{U}}%
\newcommand{\T}{\mathbb{T}}%
\newcommand{\N}{\mathbb{N}}%
\newcommand{\Z}{\mathbb{Z}}%
\newcommand{\R}{\mathbb{R}}%
\newcommand{\ep}{\varepsilon}%
\newcommand{\tm}{\times}%
\newcommand{\rmd}{\mathrm{d}}%
\newcommand{\rme}{\mathrm{e}}%
\newcommand{\rmO}{\mathrm{O}}%
\newcommand{\ess}{\operatorname*{ess\;sup}}%
\newcommand{\inv}{\operatorname{inv}}%

\begin{document}

\title{On the structure of uniformly hyperbolic chain control sets}%

\author{Christoph Kawan\footnote{Universit\"at Passau, Fakult\"at f\"ur Informatik und Mathematik, Innstra{\ss}e 33, 94032 Passau, Germany; christoph.kawan@uni-passau.de; Phone: +49(0)851 509 3363}}%
\date{}%
\maketitle%

\begin{abstract}
We prove the following theorem: Let $Q$ be an isolated chain control set of a control-affine system on a smooth compact manifold $M$. If $Q$ is uniformly hyperbolic without center bundle, then the lift of $Q$ to the extended state space $\UC\tm M$, where $\UC$ is the space of control functions, is a graph over $\UC$. In other words, for every control $u\in\UC$ there is a unique $x\in Q$ such that the corresponding state trajectory $\varphi(t,x,u)$ evolves in $Q$.%
\end{abstract}

{\small {\bf Keywords:} Nonlinear control; control-affine system; chain control set; uniform hyperbolicity}

{\small {\emph AMS Subject Classification (2010)}: 93C10; 93C15; 93B05; 37D05; 37D20}%

\section{Introduction}%

The notion of a uniformly hyperbolic set, which axiomatizes the geometric picture behind the ``horseshoe'', a general mechanism for producing complicated dynamics, was introduced by Smale in the 1960s. A uniformly hyperbolic set of a diffeomorphism $g:M\rightarrow M$ on a compact Riemannian manifold $M$ is a closed invariant set $\Lambda$ such that the tangent bundle over $\Lambda$ splits into two subbundles, $T\Lambda = E^s \oplus E^u$, invariant under the differential $\rmd g$ with uniform exponential contraction (expansion) on $E^s$ ($E^u$). For a flow $(\phi_t)_{t\in\R}$, generated by an ordinary differential equation $\dot{x} = f(x)$, a uniformly hyperbolic set is defined differently, because for any trajectory bounded away from equilibria, the vector $f(x) \in T_xM$ is neither contracted nor expanded exponentially. In this case, a uniformly hyperbolic set is a closed invariant set $\Lambda$ such that $T\Lambda = E^s \oplus E^c \oplus E^u$ with three invariant subbundles, where additionally to the contracting and expanding bundles the one-dimensional center bundle $E^c$ corresponds to the flow direction. Without the center bundle $E^c$ in this definition, a flow could only have trivial uniformly hyperbolic sets, consisting of finitely many equilibria.%

The situation looks different for systems generated by equations with explicitly time-dependent right-hand sides. General models for such systems are skew-products, which are dynamical systems of the form $\Phi:\T \tm B \tm M \rightarrow B \tm M$, $\Phi_t(b,x) = (\theta_tb,\varphi(t,x,b))$, with a time set $\T\in\{\Z,\R\}$. The solutions of the equation are incorporated in the map $\varphi$, while $\theta$ is a `driving system' on a base space $B$ that models the time-dependency of the equation. Every non-autonomous difference equation $x_{t+1} = f(t,x_t)$ or differential equation $\dot{x} = f(t,x)$ with unique and globally defined solutions gives rise to a skew-product, where $B = \T$ and $\theta_t(s) = t + s$. Other examples with less trivial base dynamics are random dynamical systems and control-affine systems. If $B$ is a compact space, $M$ a smooth manifold and $\Phi$ respects these structures, a uniformly hyperbolic set can be defined as a compact $\Phi$-invariant set $\Lambda \subset B \tm M$ such that for every $(b,x)\in\Lambda$ the tangent space $T_xM$ splits into subspaces $E^s_{b,x} \oplus E^u_{b,x}$ depending on $b$ and $x$. The invariance of the splitting now means that $\rmd\varphi_{t,b}(x)E^{s/u}_{b,x} = E^{s/u}_{\Phi_t(b,x)}$, and contraction (expansion) rates should be uniformly bounded in $b$ and $x$. One major difference to the autonomous situation is that there can exist non-trivial uniformly hyperbolic sets (whose projection to $M$ has nonempty interior) in the continuous-time case without the existence of a one-dimensional center subbundle. This, for instance, happens in random dynamical systems that arise as small time-dependent perturbations of a flow around a hyperbolic equilibrium (cf.~\cite{Liu} for the discrete-time case).%

In this paper, we consider a special type of skew-product flow, namely the control flow generated by a control-affine system, i.e., a control system governed by differential equations of the form%
\begin{equation*}
  \Sigma:\ \dot{x}(t) = f_0(x(t)) + \sum_{i=1}^mu_i(t)f_i(x(t)),\quad u \in \UC.%
\end{equation*}
The set $\UC$ of admissible control functions consists of all measurable $u:\R\rightarrow\R^m$ with values in a compact and convex set $U\subset\R^m$, and $f_0,f_1,\ldots,f_m$ are $\CC^1$-vector fields on a smooth manifold $M$. The set $\UC$, endowed with the weak$^*$-topology of $L^{\infty}(\R,\R^m) = L^1(\R,\R^m)^*$, is a compact metrizable space. For each $u\in\UC$ and $x\in M$ a unique solution to the corresponding equation exists with initial value $x$ at time $t=0$. Writing $\varphi(\cdot,x,u)$ for this solution and assuming that all such solutions exist on $\R$, one obtains a continuous skew-product flow%
\begin{equation*}
  \Phi:\R \tm \UC \tm M \rightarrow \UC \tm M,\quad \Phi_t(u,x) = (\theta_tu,\varphi(t,x,u)),%
\end{equation*}
where $\theta_tu(s) = u(t+s)$ is the shift flow on $\UC$. There are remarkable relations between dynamical properties of $\Phi$ and control-theoretic properties of $\Sigma$, a comprehensive study of which can be found in \cite{CKl}. In particular, the notions of control and chain control sets are to mention here. Control sets are the maximal subsets of $M$ on which complete approximate controllability holds. Their lifts to $\UC \tm M$ are maximal topologically transivite sets of $\Phi$. In contrast, chain control sets are the subsets of $M$ whose lifts are the maximal invariant chain transitive sets of $\Phi$, and they can be seen as an outer approximation of the control sets, since under mild assumptions a control set is contained in a chain control set.%

The purpose of this paper is to prove a theorem about the structure of a chain control set $Q$ with a uniformly hyperbolic structure without center bundle. We show that the lift of such $Q$, defined by%
\begin{equation*}
  \QC := \left\{(u,x) \in \UC \tm M\ :\ \varphi(\R,x,u) \subset Q\right\},%
\end{equation*}
has the property that each fiber $\{x\in M : (u,x) \in \QC\}$ is a singleton. In other words, $\QC$ is the graph of a (necessarily continuous) function $\UC \rightarrow Q$. This simple structure can be seen as an analogue to the fact that a connected uniformly hyperbolic set of a flow without center bundle consists of a single equilibrium. Nevertheless, from the control-theoretic viewpoint uniformly hyperbolic chain control sets are not trivial, since they can have nonempty interior and in this case are the closures of control sets (cf.~\cite{CDu,DK1}).%

The paper is organized as follows. In Section \ref{sec_mzh} we review the shadowing lemma proved in \cite{MZh} for uniformly hyperbolic sets of general skew-product maps. This is the main tool for the proof of our theorem, which is carried out in Section \ref{sec_mr}. The final Section \ref{sec_conseq} contains an application to invariance entropy.%

\section{A shadowing lemma for skew-product maps}\label{sec_mzh}%

In this section, we explain the contents of the shadowing lemma for skew-product maps proved in \cite{MZh} by Meyer and Zhang. Let $M$ be a Riemannian manifold (with metric $d(\cdot,\cdot)$) and $B$ a compact metric space. Suppose that%
\begin{equation*}
  \Phi:B \tm M \rightarrow B \tm M,\quad \Phi(b,x) = (\theta(b),\varphi(b,x)),%
\end{equation*}
is a homeomorphism such that also $\theta:B \rightarrow B$ is a homeomorphism.\footnote{In \cite{MZh}, $\theta$ is assumed to be almost periodic. However, this is not used for the proof of the shadowing lemma.} For fixed $b\in B$ assume that $\varphi_b := \varphi(b,\cdot):M \rightarrow M$ is a diffeomorphism whose derivative depends continuously on $(b,x)$. The orbit through $(b,x)$ is the set $\rmO(b,x) = \left\{\Phi^k(b,x) : k \in \Z\right\}$. We write $\varphi(k,x,b)$ for the second component of $\Phi^k(b,x)$, i.e., $\Phi^k(b,x) = (\theta^k(b),\varphi(k,x,b))$. A sequence $(b_k,x_k)_{k\in\Z}$ in $B \tm M$ is an $\alpha$-pseudo-orbit if%
\begin{equation*}
  b_{k+1} = \theta(b_k) \mbox{\quad and\quad} d(\varphi(b_k,x_k),x_{k+1}) < \alpha \mbox{\quad for all\ } k\in\Z.%
\end{equation*}
A pseudo-orbit $(b_k,x_k)_{k\in\Z}$ is $\beta$-shadowed by an orbit $\rmO(b,x)$ if%
\begin{equation*}
  b = b_0 \mbox{\quad and\quad} d(\varphi(k,x,b),x_k) < \beta \mbox{\quad for all\ } k\in\Z.%
\end{equation*}
A set $\Lambda \subset M \tm B$ is \emph{invariant} if $\Phi(\Lambda) = \Lambda$. A closed invariant set $\Lambda$ is \emph{isolated} if there exists a neighborhood $U$ of $\Lambda$ such that $\Phi^k(b,x) \in \cl U$ for all $k\in\Z$ implies $(b,x)\in\Lambda$. A closed invariant set $\Lambda$ is \emph{uniformly hyperbolic} if there are constants $C>0$, $0 < \mu < 1$ and a continuous map $(b,x) \mapsto P(b,x) \in P(T_xM,T_xM)$, defined on $\Lambda$, where $P(T_xM,T_xM)$ denotes the space of all linear projections on $T_xM$, such that%
\begin{enumerate}
\item[(i)] $P(\Phi(b,x))\rmd\varphi_b(x) = \rmd\varphi_b(x)P(b,x)$.%
\item[(ii)] $\|\rmd\varphi_{k,b}(x)P(b,x)\| \leq C\mu^k$ for all $(b,x)\in\Lambda$, $k\geq0$.%
\item[(iii)] $\|\rmd\varphi_{k,b}(x)(I-P(b,x))\| \leq C\mu^{-k}$ for all $(b,x)\in\Lambda$, $k\leq0$.%
\end{enumerate}
Here $\varphi_{k,b} = \varphi(k,\cdot,b)$. A reduced version of the shadowing lemma \cite[Lem.~2.11]{MZh} reads as follows.%

\begin{lemma}\label{lem_SL}
Let $\Lambda \subset B \tm M$ be a compact invariant uniformly hyperbolic set. Then there is a neighborhood $U$ of $\Lambda$ such that the following holds:%
\begin{enumerate}
\item[(i)] For any $\beta>0$ there is an $\alpha>0$ such that every $\alpha$-pseudo-orbit $(b_k,x_k)_{k\in\Z}$ in $U$ is $\beta$-shadowed by an orbit $\{\Phi^k(b_0,y) : k\in\Z\}$.%
\item[(ii)] There is $\beta_0>0$ such that $0 < \beta < \beta_0$ implies that the shadowing orbit in (i) is uniquely determined by the pseudo-orbit.%
\item[(iii)] If $\Lambda$ is an isolated invariant set of $\Phi$, then the shadowing orbit is in $\Lambda$.%
\end{enumerate}
\end{lemma}

\section{The main result}\label{sec_mr}%

\subsection{Preliminaries and assumptions}%

Consider a control-affine system%
\begin{equation*}
  \Sigma:\ \dot{x}(t) = f_0(x(t)) + \sum_{i=1}^m u_i(t)f_i(x(t)),\quad u\in\UC = L^{\infty}(\R,U),%
\end{equation*}
on a compact Riemannian manifold $M$ with distance $d(\cdot,\cdot)$. The vector fields $f_0,f_1,\ldots,f_m$ are assumed to be of class $\CC^1$ and the control range $U\subset\R^m$ is compact and convex. The set $\UC$ of admissible control functions is endowed with the weak$^*$-topology of $L^{\infty}(\R,\R^m) = L^1(\R,\R^m)^*$. We write%
\begin{equation*}
  \Phi:\R \tm \UC \tm M \rightarrow \UC \tm M,\quad \Phi_t(u,x) = (\theta_tu,\varphi(t,x,u)),%
\end{equation*}
for the associated control flow and $\varphi_{t,u} = \varphi(t,\cdot,u)$. A chain control set is a set $Q \subset M$ with the following properties:%
\begin{enumerate}
\item[(i)] For every $x\in Q$ there exists $u\in\UC$ with $\varphi(\R,x,u) \subset Q$.%
\item[(ii)] For each two $x,y\in Q$ and all $\ep,T>0$ there are $n\in\N$, controls $u_0,\ldots,u_{n-1}\in\UC$, states $x_0 = x,x_1,\ldots,x_{n-1},x_n = y$ and times $t_0,\ldots,t_{n-1}\geq T$ such that%
\begin{equation*}
  d(\varphi(t_i,x_i,u_i),x_{i+1}) < \ep,\quad i = 0,1,\ldots,n-1.%
\end{equation*}
\item[(iii)] $Q$ is maximal with (i) and (ii) in the sense of set inclusion.%
\end{enumerate}

Throughout the paper, we fix a chain control set $Q$ and write%
\begin{equation*}
  \QC = \left\{ (u,x) \in \UC\tm M\ :\ \varphi(\R,x,u) \subset Q \right\}%
\end{equation*}
for its \emph{full-time lift}, which is a chain recurrent component of the control flow $\Phi$ on $\UC\tm M$ (cf.~\cite[Thm.~4.1.4]{CKl}). We further assume that $\QC$ is an isolated invariant set for $\Phi$, i.e., there exists a neighborhood $N \subset \UC \tm M$ of $\QC$ such that $\Phi(\R,u,x) \subset N$ implies $(u,x) \in \QC$. This, e.g., is the case if there are only finitely many chain control sets on $M$, because then the chain recurrent components are the elements of a finest Morse decomposition. We further assume that%
\begin{equation*}
  T_xM = E^+_{u,x} \oplus E^-_{u,x},\quad \forall (u,x)\in\QC,%
\end{equation*}
with subspaces $E^{\pm}_{u,x}$ satisfying%
\begin{enumerate}
\item[(H1)] $\rmd\varphi_{t,u}(x)E^{\pm}_{u,x} = E^{\pm}_{\Phi_t(u,x)}$ for all $(u,x) \in \QC$ and $t\in\R$.%
\item[(H2)] There exist constants $0 < c \leq 1$ and $\lambda>0$ such that for all $(u,x)\in\QC$,%
\begin{equation*}
  \left|\rmd\varphi_{t,u}(x)v\right| \leq c^{-1}\rme^{-\lambda t}|v| \mbox{\quad for all\ } t\geq0,\ v\in E^-_{u,x},%
\end{equation*}
and%
\begin{equation*}
  \left|\rmd\varphi_{t,u}(x)v\right| \geq c\rme^{\lambda t}|v| \mbox{\quad for all\ } t\geq0,\ v\in E^+_{u,x}.%
\end{equation*}
\end{enumerate}
From (H1) and (H2) it follows that $E^{\pm}_{u,x}$ depend continuously on $(u,x)$ (cf.~\cite[Lem.~6.4]{Kaw}). We define the $u$-fiber of $\QC$ by%
\begin{equation*}
  Q(u) := \left\{ x\in Q\ :\ (u,x)\in\QC \right\}.%
\end{equation*}
On $\UC$ we fix a metric, compatible with the weak$^*$-topology, of the form%
\begin{equation*}
  d_{\UC}(u,v) = \sum_{n=1}^{\infty}\frac{1}{2^n}\frac{|\int_{\R}\langle u(t)-v(t),x_n(t)\rangle\rmd t|}{1 + |\int_{\R}\langle u(t)-v(t),x_n(t)\rangle\rmd t|},%
\end{equation*}
where $\{x_n : n\in\N\}$ is a dense and countable subset of $L^1(\R,\R^m)$ and $\langle\cdot,\cdot\rangle$ is a fixed inner product on $\R^m$ (cf.~\cite[Lem.~4.2.1]{CKl}).%

\subsection{Statement of results and proofs}%

We observe that the time-$1$-map $\Phi_1:\UC \tm M \rightarrow \UC \tm M$, $(u,x) \mapsto (\theta_1u,\varphi(1,x,u))$, of the control flow is a skew-product map and $\QC$ is a uniformly hyperbolic set for $\Phi_1$ in the sense of Section \ref{sec_mzh}. Moreover, $\QC$ is isolated for $\Phi_1$, which easily follows from our assumption that $\QC$ is an isolated invariant set of the control flow. Continuous dependence of the derivative $\rmd\varphi_{1,u}(x)$ on $(u,x)$ is proved in \cite[Thm.~1.1]{Kaw}.%

\begin{proposition}
For any $u,v \in \UC$ the fibers $Q(u)$ and $Q(v)$ are homeomorphic.%
\end{proposition}

\begin{proof}
The proof is subdivided into three steps.%

\emph{Step 1.} We claim that for every $\ep>0$ there is $\delta>0$ such that for all $u,v\in\UC$,%
\begin{equation}\label{eq1}
  \|u - v\|_{\infty} < \delta \quad\Rightarrow\quad d_{\UC}(\theta_tu,\theta_tv) < \ep \mbox{\quad for all\ } t\in\R,%
\end{equation}
where $\|\cdot\|_{\infty}$ is the $L^{\infty}$-norm. To prove this, choose $N = N(\ep) \in \N$ with%
\begin{equation*}
  \sum_{n=N+1}^{\infty}\frac{1}{2^n} < \frac{\ep}{2}.%
\end{equation*}
Then put%
\begin{equation*}
  c(\ep) := \max_{1 \leq n \leq N}\|x_n\|_1,\quad \delta := \frac{\ep}{2c(\ep)},%
\end{equation*}
where $\|\cdot\|_1$ is the $L^1$-norm. Then, for every $t\in\R$, $\|u-v\|_{\infty}<\delta$ implies%
\begin{eqnarray*}
  d_{\UC}(\theta_tu,\theta_tv) &=& \sum_{n=1}^{\infty}\frac{1}{2^n}\frac{|\int_{\R}\langle u(t+s)-v(t+s),x_n(s) \rangle \rmd s|}{1 + |\int_{\R}\langle u(t+s)-v(t+s),x_n(s) \rangle \rmd s|}\\
	                             &<& \sum_{n=1}^N\frac{1}{2^n}\left|\int_{\R}\langle u(t+s)-v(t+s),x_n(s) \rangle \rmd s\right| + \frac{\ep}{2}\\
															 &\leq& \sum_{n=1}^N\frac{1}{2^n}\int_{\R} |u(t+s)-v(t+s)|\cdot |x_n(s)|\rmd s + \frac{\ep}{2}\\
															 &\leq& \delta \sum_{n=1}^N\frac{1}{2^n} \max_{1\leq n\leq N}\int_{\R}|x_n(s)|\rmd s + \frac{\ep}{2} < \frac{\ep}{2 c(\ep)}c(\ep) + \frac{\ep}{2} = \ep.%
\end{eqnarray*}

\emph{Step 2.} Consider the time-$1$-map $\Phi_1$ of the control flow with the uniformly hyperbolic set $\QC$. Let $\beta>0$ be given and choose $\alpha = \alpha(\beta)$ according to the shadowing lemma \ref{lem_SL}. Then choose $\ep = \ep(\alpha)$ such that%
\begin{equation}\label{eq2}
  d_{\UC}(u,v) < \ep \quad\Rightarrow\quad d(\varphi(1,x,u),\varphi(1,x,v)) < \alpha,%
\end{equation}
whenever $u,v\in\UC$ and $x\in Q$. This is possible by uniform continuity of $\varphi(1,\cdot,\cdot)$ on the compact set $Q \tm \UC$. We claim that for all sufficiently small $\beta$,%
\begin{equation}\label{eq3}
  \sup_{t\in\R}d_{\UC}(\theta_tu,\theta_tv) < \ep \quad\Rightarrow\quad Q(u) \mbox{ and } Q(v) \mbox{ are homeomorphic}.%
\end{equation}
If $Q(u)$ and $Q(v)$ are both empty, there is nothing to show. Otherwise, we may assume that $Q(u) \neq \emptyset$. Then choose $x\in Q(u)$ arbitrarily and consider the doubly infinite sequence $x_n := \varphi(n,x,u)$, $n\in\Z$, which is completely contained in $Q$ and (by \eqref{eq2}) satisfies%
\begin{equation*}
  d(\varphi(1,x_n,\theta_nv),x_{n+1}) = d(\varphi(1,x_n,\theta_nv),\varphi(1,x_n,\theta_nu)) < \alpha%
\end{equation*}
for all $n\in\Z$. Hence, $(x_n,\theta_nv)_{n\in\Z}$ is an $\alpha$-pseudo-orbit for $\Phi_1$. By the shadowing lemma there exists $y\in M$ with%
\begin{equation*}
  d(\varphi(n,y,v),\varphi(n,x,u)) < \beta \mbox{\quad for all\ } n\in\Z.%
\end{equation*}
Since $\QC$ is isolated, Lemma \ref{lem_SL}(iii) implies $(v,y) \in \QC$, i.e., $y \in Q(v)$. We claim that the map%
\begin{equation*}
  h_{uv}:Q(u) \rightarrow Q(v),\quad x \mapsto y,%
\end{equation*}
defined in this way, is a homeomorphism. If $\beta$ is small enough, the shadowing orbit is unique by Lemma \ref{lem_SL}(ii), and hence $h_{uv}$ is uniquely defined. Since the $\beta$-shadowing relation between $u$-orbits and $v$-orbits is symmetric, one can equivalently define a map $h_{vu}:Q(v) \rightarrow Q(u)$, which by uniqueness must be the inverse of $h_{uv}$. The proof for the continuity of $h_{uv}$ is standard and will be omitted. (Continuity will also follow trivially from Proposition \ref{prop_finitefiber}).%

\emph{Step 3.} From convexity of $U$ it follows that for arbitrary $u,v\in\UC$ the curve%
\begin{equation*}
  w:[0,1] \rightarrow \UC,\quad \tau \mapsto w_{\tau},\quad w_{\tau}(t) :\equiv (1-\tau)u(t) + \tau v(t),%
\end{equation*}
is well-defined. It is continuous w.r.t.~the $L^{\infty}$-topology on $\UC$, since%
\begin{eqnarray*}
  \|w_{\tau_1} - w_{\tau_2}\|_{\infty} &=& \ess_{t\in\R}\left|(1-\tau_1)u(t) + \tau_1v(t) - (1-\tau_2)u(t) - \tau_2v(t)\right|\\
	                                     &=& \ess_{t\in\R}\left|(\tau_2-\tau_1)u(t) - (\tau_2 - \tau_1)v(t)\right|\\
																			 &=& \ess_{t\in\R} |\tau_2-\tau_1| \cdot |u(t)-v(t)| \leq |\tau_2-\tau_1|\diam U.%
\end{eqnarray*}
Hence, for each $\tau\in[0,1]$ we can pick a (relatively) open subinterval $I_{\tau} \subset [0,1]$ containing $\tau$ such that $\|w_s - w_r\|_{\infty}$ is smaller than a given constant for all $s,r\in I_{\tau}$. By Step 1 this implies that $\sup_{t\in\R}d_{\UC}(\theta_t w_s,\theta_t w_r) < \ep$ for all $s,r\in I_{\tau}$ with a given $\ep>0$. Choose $\ep$ according to Step 2 so that $Q(w_s)$ and $Q(w_r)$ are homeomorphic for any two $s,r\in I_{\tau}$. By compactness, finitely many such intervals $I_{\tau_1},\ldots,I_{\tau_l}$ are sufficient to cover $[0,1]$. We may assume that $0 = \inf I_{\tau_1} < \inf I_{\tau_2} < \cdots < \inf I_{\tau_l} < \sup I_{\tau_l} = 1$. To show that $Q(u)$ and $Q(v)$ are homeomorphic, we put $t_0 := 0$, $t_l := 1$ and pick $t_i \in I_{\tau_i} \cap I_{\tau_{i+1}} \neq \emptyset$ for $i = 1,\ldots,l-1$. Then there exist homeomorphisms $h_{i,i+1}:Q(w_{t_i}) \rightarrow Q(w_{t_{i+1}})$ for $0 \leq i \leq l-1$. The composition of these homeomorphisms gives a homeomorphism from $Q(u) = Q(w_{t_0})$ to $Q(v) = Q(w_{t_l})$.%
\end{proof}

\begin{corollary}\label{cor_ops}
If $Q(u)$ is a singleton for one $u\in\UC$, then $\QC$ is the graph of a continuous map from $\UC$ to $Q$.%
\end{corollary}

\begin{proof}
By the proposition, $Q(u)$ is a singleton for every $u\in\UC$, say $Q(u) = \{x(u)\}$. Consider the map $u \mapsto (u,x(u))$, $\UC\rightarrow\QC$. This is an invertible map between compact metric spaces with (obviously) continuous inverse. Hence, it is a homeomorphism. This implies that the function $u \mapsto x(u)$, $\UC\rightarrow Q$, is continuous and $\QC$ is its graph.%
\end{proof}

The next proposition shows that the fibers $Q(u)$ are finite.%

\begin{proposition}\label{prop_finitefiber}
If $u$ is a constant control function, then $Q(u)$ consists of finitely many equilibria. Hence, there exists $n\in\N$ such that $Q(u)$ has precisely $n$ elements for every $u\in\UC$.%
\end{proposition}

\begin{proof}
Let $u\in\UC$ be a constant control function. Observe that $Q(u)$ is a uniformly hyperbolic set for the diffeomorphism $g := \varphi_{1,u}:M\rightarrow M$. It is well-known that a diffeomorphism is expansive on a uniformly hyperbolic set, i.e., there is $\ep>0$ such that $d(g^k(x),g^k(y)) < \ep$ for all $k\in\Z$ and $x,y\in Q(u)$ implies $x=y$ (see \cite[Cor.~6.4.10]{KHa}). If $x\in Q(u)$ and $w := f_0(x) + \sum_{i=1}^m u_i f_i(x) \neq 0$, then the Lyapunov exponent $l(w) := \limsup_{t\rightarrow\infty}(1/t)\log|\rmd\varphi_{t,u}(x)w|$ vanishes if the trajectory $\varphi(t,x,u)$ is bounded away from equilibria. A zero Lyapunov exponent, however, contradicts the existence of the uniformly hyperbolic splitting on $\QC$. Since the right-hand side of the system is bounded on compact sets, $l(w)<0$ follows, implying $w \in E^-_{u,x}$. Writing $f := f_0 + \sum_{i=1}^m u_i f_i$, this yields $|\rmd\varphi_{t,u}(x)w| = |f(\varphi(t,x,u))| \leq c^{-1}\rme^{-\lambda t}$ for $t\geq0$. Because of the uniform hyperbolicity, there can be at most finitely many equilibria in the compact set $Q(u)$, and hence $\varphi(t,x,u) \rightarrow z_+$ for some equilibrium $z_+$. The same argumentation for the backward flow yields $\varphi(t,x,u) \rightarrow z_-$ for an equilibrium $z_-$. Choose $t_0>0$ large enough so that%
\begin{equation}\label{eq_t0choice}
  d(\varphi(t,x,u),z_{\pm}) < \frac{\ep}{2} \mbox{\quad if\ } |t| \geq \frac{t_0}{2}.%
\end{equation}
Then choose $\delta>0$ small enough so that%
\begin{equation}\label{eq_deltachoice}
  d(x,y) < \delta \quad\Rightarrow\quad d(\varphi(t,x,u),\varphi(t,y,u)) < \ep \mbox{\quad for all\ } |t| \leq t_0.%
\end{equation}
Finally, let $\tau \in (0,t_0/2)$ be chosen so that%
\begin{equation}\label{eq_tauchoice}
  d(x,\varphi(\tau,x,u)) < \delta.%
\end{equation}
We let $y := \varphi(\tau,x,u)$ and claim that $d(\varphi(t,x,u),\varphi(t,y,u)) < \ep$ for all $t\in\R$, implying $x = y$. Indeed, by \eqref{eq_tauchoice} and \eqref{eq_deltachoice} we have%
\begin{equation*}
  d(\varphi(t,x,u),\varphi(t,y,u)) < \ep \mbox{\quad for all\ } |t| \leq t_0.%
\end{equation*}
Now assume that $t \geq t_0$. Then \eqref{eq_t0choice} yields%
\begin{equation*}
  d(\varphi(t,x,u),\varphi(t,y,u)) \leq d(\varphi(t,x,u),z_+) + d(z_+,\varphi(t+\tau,x,u)) < \ep.%
\end{equation*}
If $t < -t_0$, we obtain $t + \tau < -t_0 + \tau < -t_0 + t_0/2 = -t_0/2$, and hence \eqref{eq_t0choice} gives%
\begin{equation*}
  d(\varphi(t,x,u),\varphi(t,y,u)) \leq d(\varphi(t,x,u),z_-) + d(z_-,\varphi(t+\tau,x,u)) < \ep.%
\end{equation*}
In particular, $d(g^k(x),g^k(y)) < \ep$ for all $k\in\Z$, and hence $x=y=z_+=z_-$. Consequently, $Q(u)$ consists of finitely many equilibria.%
\end{proof}

The next theorem is our main result.%

\begin{theorem}\label{thm_mainres}
Consider the control-affine system $\Sigma$ with the uniformly hyperbolic chain control set $Q$ with isolated lift $\QC$. Assume that $\inner U \neq \emptyset$ and let $u_0$ be a constant control function with value in $\inner U$. Additionally suppose that the following hypotheses are satisfied:%
\begin{enumerate}
\item[(i)] The vector fields $f_0,f_1,\ldots,f_m$ are of class $\CC^{\infty}$ and the Lie algebra generated by them has full rank at each point of $Q$.%
\item[(ii)] For each $x\in Q(u_0)$ and each $\rho \in (0,1]$ it holds that $x \in \inner\OC^+_{\rho}(x)$, where $\OC^+_{\rho}(x) = \{\varphi(t,x,u) : t\geq 0,\ u\in\UC^{\rho}\}$ with
\begin{equation*}
  \UC^{\rho} = \{u \in \UC\ :\ u(t) \in u_0 + \rho(U - u_0) \mbox{ a.e.}\}.%
\end{equation*}
\end{enumerate}
Then $\QC$ is a graph of a continuous function $\UC \rightarrow Q$.%
\end{theorem}

\begin{remark}
Before proving the theorem, we note that assumption (ii) is in particular satisfied if the system is locally controllable at $(u_0,x)$ for each $x \in Q(u_0)$ (using arbitrarily small control ranges around $u_0$). A sufficient condition for this to hold, which is independent of $\rho$, is the controllability of the linearization around $(u_0,x)$.%
\end{remark}

\begin{proof}
Without loss of generality, we assume that $u_0(t) \equiv 0$. Let $Q(u_0) = \{x_1,\ldots,x_n\}$. We consider for each $\rho \in (0,1]$ the control-affine system%
\begin{equation*}
  \Sigma^{\rho}:\ \dot{x}(t) = f_0(x(t)) + \sum_{i=1}^m u_i(t)f_i(x(t)),\quad u\in\UC^{\rho}.%
\end{equation*}
From the assumptions (i) and (ii) it follows by \cite[Cor.~4.1.7]{CKl} that each $x_i$ is contained in the interior of a control set $D_i^{\rho}$ of $\Sigma^{\rho}$. Each $D_i^{\rho}$ is contained in a unique chain control set $E_i^{\rho}$ of $\Sigma^{\rho}$ (cf.~\cite[Cor.~4.3.12]{CKl}). By \cite[Cor.~3.1.14]{CKl} the chain control sets depend upper semicontinuously on $\rho$, hence $E_i^{\rho} \subset Q = E_i^1$ for every $1 \leq i \leq n$ and $\rho\in(0,1]$. This implies that each $E_i^{\rho}$ is uniformly hyperbolic. By \cite[Thm.~3]{CDu} it follows that $E_i^{\rho} = \cl D_i^{\rho}$. If $C_i$ denotes the chain recurrent component of the uncontrolled system $\dot{x} = f_0(x)$ which contains the equilibrium $x_i$, then $C_i \subset E_i^{\rho}$ for each $\rho$, because otherwise $E_i^{\rho} \cup C_i$ would satisfy the first two properties of chain control sets, contradicting maximality of $E_i^{\rho}$. Since each chain recurrent component is connected and $C_i \subset Q(u_0)$, we have $C_i = \{x_i\}$. By \cite[Cor.~3.4.10]{CKl}, the chain control set $E_i^{\rho}$ shrinks to $\{x_i\}$ as $\rho\searrow0$. Hence, for small $\rho$, the sets $E_i^{\rho}$ are pairwisely disjoint. Since $E_i^1 = Q$ for each $i$, at some point the chain control sets have to merge as $\rho$ increases. Since, by \cite[Thm.~3.1.12]{CKl}, the control sets $D_i^{\rho}$ depend lower semicontinuously on $\rho$, this is a contradiction if $n>1$. It follows that $n=1$ and Corollary \ref{cor_ops} yields the assertion.%
\end{proof}

\begin{remark}
Of course, in many cases it will be easier to check directly that $Q(u)$ is a single equilibrium for some constant control function $u$ than verifying the conditions of the preceding theorem. We also note that the fact that $\QC$ is a graph over $\UC$ implies the existence of a topological conjugacy between the shift flow $\theta$ on $\UC$ and the restriction of the control flow to $\QC$ (cf.~\cite{DK1}).%
\end{remark}

\section{Application to invariance entropy}\label{sec_conseq}%

The invariance entropy of a controlled invariant subset $Q$ of $M$ measures the complexity of the control task of keeping the state inside $Q$. In general, it is defined as follows. A pair $(K,Q)$ of subsets of $M$ is called admissible if $K$ is compact and for every $x\in K$ there is $u\in\UC$ with $\varphi(\R_+,x,u) \subset Q$. In particular, if $K=Q$, this means that $Q$ is a compact and controlled invariant set. For $\tau>0$, a set $\SC\subset\UC$ is $(\tau,K,Q)$-spanning if for every $x\in K$ there is $u\in\SC$ with $\varphi([0,\tau],x,u) \subset Q$. Then $r_{\inv}(\tau,K,Q)$ denotes the number of elements in a minimal such set and we put $r_{\inv}(\tau,K,Q) := \infty$ if no finite $(\tau,K,Q)$-spanning set exists. The \emph{invariance entropy} of $(K,Q)$ is%
\begin{equation*}
  h_{\inv}(K,Q) := \limsup_{\tau\rightarrow\infty}\frac{1}{\tau}\log r_{\inv}(\tau,K,Q),%
\end{equation*}
where $\log$ is the natural logarithm. From \cite[Thm.~5.4]{DK1} we can conclude the following result on the invariance entropy of admissible pairs $(K,Q)$, where $Q$ is a uniformly hyperbolic chain control set. The difference to \cite[Thm.~5.4]{DK1} is that we do not have to assume explicitly anymore that $\QC$ is a graph over $\UC$.%

\begin{theorem}
Consider the control-affine system $\Sigma$ with the uniformly hyperbolic chain control set $Q$ with isolated lift $\QC$. Let the assumptions (i) and (ii) of Theorem \ref{thm_mainres} be satisfied, or alternatively, assume that $Q(u)$ is a singleton for some $u\in\UC$. Then $Q$ is the closure of a control set $D$ and for every compact set $K\subset D$ of positive volume the pair $(K,Q)$ is admissible and its invariance entropy satisfies%
\begin{equation*}
  h_{\inv}(K,Q) = \inf_{(u,x)\in\QC}\limsup_{\tau\rightarrow\infty}\frac{1}{\tau}\log\left|\det(\rmd\varphi_{\tau,u})|_{E^+_{u,x}}:E^+_{u,x}\rightarrow E^+_{\Phi_{\tau}(u,x)}\right|.%
\end{equation*}
\end{theorem}

\begin{remark}
The paper \cite{DK2} provides a rich class of examples for uniformly hyperbolic chain control sets that arise on the flag manifolds of a semisimple Lie group. The control-affine system in this case is induced by a right-invariant system on the group.%
\end{remark}

\end{document}